\documentclass[12pt]{article}
\usepackage[margin=2.5cm]{geometry}
\usepackage{amssymb,amsmath}              
\usepackage{graphicx}
\usepackage{amssymb}
\usepackage{color}
\usepackage{epstopdf}
\renewcommand{\epsilon}{\varepsilon}
\DeclareMathOperator{\dvg}{div} 
\DeclareMathOperator{\dist}{dist} 
 
\DeclareMathOperator{\diam}{diam}
  
  \DeclareMathOperator{\proj}{proj}
  \DeclareMathOperator{\area}{Area}
  \DeclareMathOperator{\length}{Length}
    \DeclareMathOperator{\inj}{Inj}
    
\def\R{\mathbb{R}}

\def\d{\delta}
\def\a{\alpha}
\def\e{\epsilon}

\def\r{\rho}

\def\s{\sigma}
\def\o{\omega}
\def\l{\lambda}

\def\wt{\widetilde}

\def\B{\mathcal{B}}

\def\Si{\Sigma}

\newtheorem{theorem}{Theorem}[section]
\newtheorem{lemma}[theorem]{Lemma}

\newtheorem{remark}[theorem]{Remark}
\newtheorem{corollary}[theorem]{Corollary}

\newcommand{\qed}{\hfill \ensuremath{\Box}}

\newenvironment{proof}{\smallskip\noindent{\it Proof.}\hskip \labelsep}
                          {\hfill\penalty10000\raisebox{-.09em}{$\Box$}\par\medskip}

\begin{document}

\begin{title}
{Curvature estimates for surfaces with bounded mean curvature}
\end{title}
\begin{author}
{Theodora Bourni \and Giuseppe Tinaglia~\footnote{Partially supported by The Leverhulme Trust and EPSRC grant no. EP/I01294X/1}}
\end{author}
\date{}
\maketitle
\vspace{-1cm}

\begin{abstract}
Estimates for the norm of the second fundamental form, $|A|$, play a crucial role in studying the geometry of surfaces in $\mathbb{R}^3$. In fact, when $|A|$ is bounded the surface cannot bend too sharply. In this paper we prove that for an embedded geodesic disk with bounded $L^2$ norm of $|A|$, $|A|$ is bounded at interior points, provided that the $W^{1,p}$ norm of its mean curvature is sufficiently small, $p>2$. In doing this we generalize some renowned estimates on $|A|$ for minimal surfaces. \end{abstract}

\section{Introduction.}

In the study of the geometry of surfaces in $\mathbb{R}^3$, estimates for the norm of the second fundamental form, $|A|$, are particularly remarkable. In fact, when $|A|$ is bounded the surface cannot bend too sharply and thus such estimates provide a very satisfying description of its local geometry. When a surface $\Sigma$ is minimal $|A|^2=-2K_\Sigma$, $K_\Sigma$ being the Gaussian curvature, and such estimates are then known as curvature estimates. There are many results in the literature where curvature estimates for minimal surfaces are obtained assuming certain geometric conditions, see for instance~\cite{cs1,cm22,cm23,he1,sc3,ss1,ssy1,tin3} et al.. In~\cite{cm23}, Colding and Minicozzi prove that an embedded geodesic minimal disk with  bounded  $L^2$ norm of $|A|$, {\it bounded total curvature}, has curvature bounded in the interior.

The main result in this paper is the following estimate that generalizes the curvature estimate in~\cite{cm23} to a broader class of surfaces.

\begin{theorem}\label{main}
Given $C_1$ and $p\ge 2$, there exist $C_2=C_2(p, C_1)\ge0$ and $\e_p=\e_p(C_1)>0$ such that the following holds. Let $\Sigma$ be a surface embedded in $\mathbb{R}^3$ containing the origin with $\inj_\Sigma(0)\geq s >0$,
\[
\int_{\B_{s}}|A|^2\le C_1
\]
and either
\begin{itemize}
\item[i.] $  \|H\|^*_{W^{2,2}(\B_{s})}\le\e_2, $ if $p=2,$ or
\item[ii.] $\|H\|^*_{W^{1,p}(\B_{s})}\le\e_p,$ if $p>2$,
\end{itemize}
then
\[
|A|^2(0)\le C_2s^{-2}.\]
\end{theorem}

Here, for any $x\in \Sigma$, $\inj_\Sigma(x)$ denotes the injectivity radius of $\Sigma$ at $x$. For any $s>0$,  $\B_s$ denotes the intrinsic ball of radius $s$ centered at the origin and $\|H\|^*_{W^{1,p}(\B_{s})}$ ($\|H\|^*_{W^{2,2}(\B_{s})}$) denotes the scale invariant $W^{1,p}$ ($W^{2,2}$ respectively) norm of the mean curvature, see beginning of Section~\ref{chs} for a precise definition.

The structure of this paper, that is the proof of Theorem \ref{main}, is as follows:
In Section~\ref{chs} we generalize the renowned curvature estimate by Choi and Schoen~\cite{cs1}, Theorem \ref{choischoen}. We also show that the hypotheses of Theorem~\ref{main} are optimal, Remark~\ref{rmk}. In Section~\ref{CMgen} we use this estimate to prove case (i) of Theorem~\ref{main} and also case (ii) of Theorem~\ref{main} but with the additional assumption that the $L^2$ norm of the mean curvature is small, Theorem~\ref{varofmain}. Finally in Section~\ref{elp}, we show some relation between the total curvature and the area of an intrinsic ball, which allows us to remove this extra assumption and thus finish the proof of case (ii) of Theorem~\ref{main}, Remark~\ref{rmk2}. This relation also enables us to replace the bound on the total curvature in Theorem~\ref{main} with an area bound, Corollary~\ref{cor}.

\section{Choi-Schoen curvature estimate generalized}\label{chs}

The Choi-Schoen curvature estimate \cite{cs1} says that if the total curvature of an intrinsic minimal disk is sufficiently small, then the curvature of the disk is bounded in the interior and it decays like the inverse square of the distance of the point to the boundary. The goal of this section is to generalize the Choi-Schoen curvature estimate.

Throughout this paper
$\|H\|^*_{W^{1,p}(\B_{s})}$ and $\|H\|^*_{W^{2,2}(\B_{s})}$ will denote the scale invariant $W^{1,p}$ and $W^{2,2}$ respectively norm of the mean curvature, i.e.
\[\|H\|^*_{W^{1,p}(\B_{s})}:=s^{p-2}\int_{\B_{s}}|H|^p+s^{2p-2}\int_{\B_{s}}|\nabla H|^p\]
and
$$\|H\|^*_{W^{2,2}(\B_{s})}:=\int_{\B_{s}}|H|^2+s^2\int_{\B_{s}}|\nabla H|^2+s^4\int_{\B_{s}}|\nabla^2 H|^2.$$
Furthermore
the letter $c$ will denote an absolute constant. When different constants appear in the course of a proof we will keep the same letter
$c$ unless the constant depends on some different parameters.

\begin{theorem}\label{choischoen}

 Given $p\ge 2$, there exists $\e_0=\e_0(p)>0$ such that the following holds. Let $\Si$ be a surface immersed in $\R^3$ containing the origin and $\B_{r_0}\subset \Si$, $r_0>0$. If there exists $\delta\in[0,1]$ such that

 $$\int_{\B_{r_0}}|A|^2\le\delta\e_0$$
and either
\begin{itemize}
\item[i.] $\|H\|^*_{W^{2,2}(\B_{r_0})}\le\delta\e_0$, if $p=2$,
 or
\item[ii.] $\|H\|^*_{W^{1,p}(\B_{r_0})}\le(\delta\e_0)^{p/2}$, if $p>2$,
\end{itemize}
then for all $0<\s\le r_0$ and $y\in \B_{r_0-\s}$
\[\s^2|A|^2(y)\le\d.\]
\end{theorem}


 In order to demonstrate Theorem \ref{choischoen}, we begin by proving certain results about manifolds that are not necessarily minimal. In particular we prove a Generalized Mean Value Property, Lemma \ref{gmean} and a Generalized Mean Value Inequality, Lemma \ref{GMVin}. See   Proposition 1.16 and Corollary 1.17 in~\cite{cmCourant} for a proof of these results in the minimal case.
 
 For any $s>0$, let $B_s$ denote the extrinsic ball of radius $s$ in $\R^n$ centered at the origin.
\begin{lemma}[Generalized Mean Value Property]\label{gmean}
Let $\Sigma$ be a $k$-dimensional manifold immersed in $\R^n$ and containing the origin and let $f$ be a non-negative $C^1$ function on $\Si$ then

\begin{equation}\label{GMVdereq}
\frac{d}{dr}\left(r^{-k}\int_{B_r\cap\Si} f\right)=\frac{d}{dr}\int_{B_r\cap\Si} f\frac{|x^N|^2}{|x|^{k+2}}+r^{-k-1}\int_{B_r\cap \Si}x\cdot(\nabla f+f H),\end{equation}
where $x^N$ denotes the normal component of $x$, and for $0<s<t$

\begin{equation}\label{GMVeq}\begin{split}t^{-k}\int_{B_t\cap\Si}f  -s^{-k}\int_{B_s\cap\Si}f= &\int_{(B_t\setminus B_s)\cap \Si} f\frac{|x^N|^2}{|x|^{k+2}}+ \int^t_sr^{-k-1}\int_{B_r\cap \Si}x\cdot(\nabla f+f H).\end{split}\end{equation}

\end{lemma}

\begin{proof}
Using the formula
\[\int \dvg_\Si X =-\int X\cdot H\]
with the vector field $X(x)=\gamma(|x|)f(x)x$, where $\gamma\in C^1(\R)$ is such that, for some $r>0$,
$\gamma(t)=1$ for $t\le r/2$, $\gamma(t)=0$ for $t\ge r$ and $\gamma'(t)\le 0$, we get
\[ \frac{d}{dr}\left(r^{-k}\int_{B_r\cap\Si}\phi\left(\frac{|x|}{r}\right)f\right)= r^{-k}\frac{d}{dr}\int_{B_r\cap\Si} f\frac{|x^N|^2}{|x|^{2}}\phi\left(\frac{|x|}{r}\right) +r^{-k-1}\int_{B_r\cap \Si}x\cdot(\nabla f+f H)\phi\left(\frac{|x|}{r}\right) \]
where $\phi:\R\to\R$ is defined by $\phi(|x|/r)=\gamma(|x|)$ (cf.  equation 18.1 in~\cite{si1}).
Then \eqref{GMVdereq} follows after letting $\phi$ in the above formula increase to the characteristic function of $(-\infty,1)$ and \eqref{GMVeq} follows by integrating \eqref{GMVdereq} from $s$ to $t$.
\end{proof}

\begin{remark}\label{GMVrmk}The leading terms on the RHS of both \eqref{GMVdereq} and \eqref{GMVeq} in \emph{Lemma \ref{GMVdereq}} are positive. For the second term on the RHS of \eqref{GMVeq} we note that
 for any $C^1$ function $h$ on $\Si$, integration by parts yields
\begin{equation*}\label{ibp trick}\int_s^t r^{-k-1}\int_{B_r\cap\Si} h=\frac 1k\int_{B_t\cap\Si}h\left(\frac{1}{r_s^k}-\frac{1}{t^k}\right)\end{equation*}
where $r_s=\max\{|x|,s\}$,
and furthermore
\[x\cdot\nabla f=\proj_{T_x\Si}x\cdot\nabla f=\frac 12 \nabla|x|^2\cdot\nabla f=-\frac 12 \nabla(r^2-|x|^2)\cdot\nabla f.\]
Thus, integrating by parts, we get the following two estimates, as a corollary of \emph{Lemma~\ref{GMVdereq}}, which we will need later:
\begin{equation}\label{gmean1}
\frac{d}{dr}\left(r^{-k}\int_{B_r\cap\Si}f\right)\ge r^{-k-1}\int_{B_r\cap\Si}f x\cdot H
+\frac 12r^{-k-1}\int_{B_r\cap \Si}(r^2-|x|^2)\Delta_\Si f
\end{equation}
and
\begin{equation}\label{gmean2}
 t^{-k}\int_{B_t\cap\Si}f  -s^{-k}\int_{B_s\cap\Si}f\ge \frac 1k\int_{B_t\cap\Si}f x\cdot H\left(\frac{1}{r_s^k}-\frac{1}{t^k}\right) +\frac 1k\int_{B_t\cap \Si}x\cdot\nabla f\left(\frac{1}{r_s^k}-\frac{1}{t^k}\right)
\end{equation}
where $r_s=r_s(x)=\max\{|x|,s\}$.
\end{remark}

\begin{remark}\label{c2 estimate}
In the case of a surface immersed in $\R^3$ we can use \emph{Remark \ref{GMVrmk}} to estimate the ratios $s^{-1}\int_{B_s\cap\Si} f$
as follows:
In inequality \eqref{gmean2} of \emph{Remark \ref{GMVrmk}} let $n=3$, $k=2$, $t=1$. After multiplying  by $s$, since $s, |x|\le r_s$ we get the following:
\begin{equation*}
s^{-1}\int_{B_s\cap\Si}f\le \int_{B_1\cap\Si}f+\frac 12\int_{B_1\cap\Si} f |H|+\frac 12 \int_{B_1\cap\Si}|\nabla f|.
\end{equation*}
\end{remark}

Using Lemma \ref{gmean} we obtain the following Generalized Mean Value Inequality.
\begin{lemma}[Generalized Mean Value Inequality]\label{GMVin}
 Let $\Sigma$ be a  hyper-surface immersed in $\R^n$ containing the origin and such that $ B_1(0)\cap\partial\Si=\emptyset$. Let also $f$ be a non-negative function on $\Si$ such that
\begin{equation}\label{laplacian}
\Delta_\Si f\ge-\l_1 f-h
\end{equation}
for some $\lambda_1\geq 0$ and a function $h\in L^1(\Si\cap B_1(0))$ satisfying the following: There exist constants $c_2, c_3$ and $\a \in [0,1)$ such that
\[\frac12r^{-n}\int_{B_r\cap \Si}(r^2-|x|^2)h\le  c_2 r^{-\a}+c_3,\,\forall r\in(0,1] .\]
Then
\begin{equation*}\label{mean value inequality}
f(0)\le\o_{n-1}^{-1} e^{c_1}\int_{B_1\cap\Si} f+\o_{n-1}^{-1}\left(\frac{c_2}{1-\a}+c_3\right)e^{c_1}
\end{equation*}
where $\o_{n-1}$ is the volume of the unit ball in $\R^{n-1}$ and
$c_1=\sup_{\Si\cap B_1(0)} |H|+\frac {\l_1}{ 2}$.
\end{lemma}

\begin{proof}
Define
\[g(t)= t^{-(n-1)}\int_{B_t\cap\Si} f\]
then using \eqref{gmean1} and \eqref{laplacian}
\[\begin{split}g'(t)&\ge -g(t) \left(\frac {\l_1}{2}t+\sup_{\Si\cap B_1(0)} |H|\right)- \frac 12t^{-n}\int_{B_t\cap\Si}(t^2-|x|^2)h\end{split}.\]
Hence, since $t\le 1$ and by the definition of $c_1, c_2, c_3$
\[g'(t)+c_1 g(t)\ge -c_3 -c_2 t^{-\a}\implies (g(t)e^{c_1 t})'\ge-{c_2} e^{c_1 }t^{-\a} -c_3 e^{c_1}.\]
After integrating from 0 to 1:
\[\o_{n-1} f(0)= \lim_{t\to 0^+}g(t)\le e^{c_1 }g(1)+\int_0^1e^{c_1}(c_2t^{-\a}+c_3)dt\le e^{c_1 }g(1)+\left(\frac{c_2}{1-\a}+c_3\right)e^{c_1}.\]
\end{proof}

In order to prove Theorem~\ref{choischoen} we still need some information about $|A|$. It is well-known that for a minimal hypersurface in $\R^n$, $|A|$ satisfies the following partial differential inequality, Simons' inequality \cite{sim1},
\begin{equation*}
\Delta|A|^2\ge -2|A|^4+2\left(1+\frac{2}{n+1}\right)|\nabla|A||^2.\end{equation*}
In \cite{echu}, Ecker and Huisken generalized Simons' inequality to obtain the following estimate for hypersurfaces in $\R^n$:
\begin{equation*}
 \Delta|A|^2\ge  2h_{ij}\nabla_i\nabla_jH-2|A|^4+2Hh_{ij}h_{ik}h_{jk} +2\left(1+\frac{2}{n+1}\right)|\nabla|A||^2-c(n)|\nabla H|^2. \end{equation*}
where the $h_{ij}$'s are the coefficients of the second fundamental form $A$. When $n=3$, $2Hh_{ij}h_{ik}h_{jk}\ge H^2|A|^2$ and thus we obtain
\begin{equation}\label{deltaA2}
\Delta|A|^2\ge 2h_{ij}\nabla_i\nabla_jH-2|A|^4-c|\nabla H|^2.
\end{equation}
Using equation~\eqref{deltaA2} and the previous results we now prove Theorem~\ref{choischoen}.\\

\noindent{\textit{Proof of Theorem~\ref{choischoen}.}}
Note first that we can assume that $\d>0$, since else the theorem is trivially true. We first prove case (i), i.e. when we assume that $\|H\|^*_{W^{2,2}(\B_{r_0})}\le\delta\e_0$. Set $F=(r_0-r)^2|A|^2$ on $\B_{r_0}$, where $r(x)=|x|$, and let $\d_0$ be the maximum value of $F$ and $x_0$ the point where this maximum is attained. Assume, for a contradiction, that $\d_0>\d$ and pick $\s$ so that
\[\s^2|A|^2(x_0)=\frac{\d}{4}.\]
Then:
\[2\s\le r_0-r(x_0)\text{  and  }\frac12\le\frac{r_0-r}{r_0-r(x_0)}\le\frac32\,\,,\,\forall x\in \B_\s(x_0)\]
and
\[(r_0-r(x_0))^2\sup_{\B_\s(x_0)}|A|^2\le 4F(x_0)\implies
\sup_{\B_\s(x_0)}|A|^2\le 4|A|^2(x_0)=\s^{-2}\d.\]

Let $\wt{\Si}=\eta_{x_0,\frac{\s}{4}}(\B_{r_0})$ (where $\eta_{x,\l}(y)=\l^{-1}(y-x) $ that is a rescaling and a translation) and let $\wt{A}$, $\wt{H}$ be the second fundamental form and the mean curvature of $\wt\Si$. Then
\begin{equation}\label{assu}\sup_{\B_{4}\subset\wt{\Si}}|\wt{A}|^2\le\frac{\s^2}{16}\sup_{\B_\s(x_0)\subset \Si}|A|^2\le\frac{\d}{16}<\frac{1}{16} \text{  and  } |\wt{A}|^2(0)=\frac{\d}{64}.\end{equation}
Note that by $\B_r$ we now denote the geodesic balls of radius $r$ in $\wt\Si$ centered at the origin.
Let $\wt\Si_0$ be the connected component  of
$B_{1}\cap \B_4$ containing the origin. Then, $\wt\Si_0$ has its boundary contained in $\partial B_1$. The proof of this is a standard argument that for completeness we have added in the Appendix, see Lemma~\ref{distlemma}. Using the Gauss equation gives
\[\sup_{\B_{4}}|K_{\wt{\Si}}|= \sup_{\B_{4}}\frac{|\wt{H}^2-|\wt{A}|^2|}{2}\leq \frac12 \left( \sup_{\B_{4}}\wt{H}^2+\sup_{\B_{4}}|\wt{A}|^2\right)\leq\frac32\sup_{\B_{4}}|\wt{A}|^2\leq \frac{3}{32}\]
 and thus, by the Bishop Volume Comparison Theorem, see for instance~\cite{chEb}, this bound implies that there exists a constant $c_b$ such that $\area{\B_{4}}<c_b$. Note that this area bound depends only on the upper bound for the absolute value of the Gaussian curvature and the radius of the intrinsic ball. Finally,
 \begin{equation}\label{areabd}
\area{\wt\Si_0}\leq \area{\B_{4}}<c_b.
\end{equation}

In what follows, we focus our analysis on $\wt\Si_0$. With an abuse of notation, we omit the tildes and set $\Si=\wt\Si_0$. Furthermore
the letter $c$ will denote an absolute constant and when different constants appear in the course of the proof we will keep the same letter
$c$ unless the constant depends on some different parameters.
We are going to show that we can apply Lemma \ref{GMVin}, with $f=|A|^2$ and $\a=\frac12$ to get an estimate for $|A(0)|$ in terms of the total curvature and $\|H\|_{W^{2,2}(\Si)}$.

The generalized Simons' inequality, with $|A|<1$ implies:
\begin{equation*}
\Delta|A|^2\ge -2|A|^2-c|\nabla H|^2 +2h_{ij}\nabla_i\nabla_jH.
\end{equation*}
Hence the inequality in the assumptions of  Lemma \ref{GMVin} is satisfied with $\lambda_1=2$ and $h=c|\nabla H|^2-2 h_{ij}\nabla_i\nabla_jH$. Furthermore we have
\[\sup_{\Si}|H|+\lambda_1/2\le \sqrt2\sup_{\Si}|A|+1\leq\sqrt2+1.\]

We are now going to find $c_2$ and $c_3$  such that
\[\frac12r^{-3}\int_{B_r\cap \Si}(r^2-|x|^2)(c|\nabla H|^2-2 h_{ij}\nabla_i\nabla_jH)\le  c_2 r^{-\frac12}+c_3,\, \text{for any } r\in(0,1]. \]
Integrating by parts we obtain
\begin{equation}\label{bad}
\begin{split}-\frac12 r^{-3}\int_{B_r\cap\Si}(r^2-|x|^2)2 h_{ij}\nabla_i\nabla_jH=
&r^{-3}\int_{B_r\cap\Si}(r^2-|x|^2)\nabla_i h_{ij}\nabla_jH\\
&+r^{-3}\int_{B_r\cap\Si}\nabla_i(r^2-|x|^2)2 h_{ij}\nabla_jH.\end{split}\end{equation}
Using Codazzi equations we can estimate the first term on the RHS of \eqref{bad} as follows:
\[\begin{split}r^{-3}\int_{B_r\cap\Si}(r^2-|x|^2)\nabla_i h_{ij}\nabla_jH\le& r^{-3}\int_{B_r\cap\Si}(r^2-|x|^2)|\nabla H|^2\end{split}.\]
We estimate the second term on the RHS of \eqref{bad}, using the fact that $\sup_\Si|A|<\sqrt{\d}/4$, as follows:
\[\begin{split}&r^{-3}\int_{B_r\cap\Si}\nabla_i(r^2-|x|^2)2 h_{ij}\nabla_jH\le r^{-3}\int_{B_r\cap\Si}2|x| |A||\nabla H|\le 2r^{-2}\int_{B_r\cap\Si}|A||\nabla H|\\
&\le \sqrt{\d}r^{-2}\int_{B_r\cap\Si}|\nabla H|\le \sqrt{\d}r^{-3/2}\area{(\Si\cap B_r(x))}^{1/2}\left(\frac 1r\int_{B_r(x)\cap\Si}|\nabla H|^2\right)^{1/2} \end{split}\]
Since $\sup_{\Si}|H|\leq\sqrt2$, the monotonicity inequality implies that there exists an absolute constant $c$ such that
\[r^{-2}\area{(\Si\cap B_r)}\le c\area{\Si}\le c c_b.\]
Thus
\begin{equation}\label{badterm}\begin{split}
\frac12r^{-3}\int_{B_r\cap \Si}&(r^2-|x|^2)(c|\nabla H|^2-2 h_{ij}\nabla_i\nabla_jH)\le\\   &c\bigg(r^{-1}\int_{B_r\cap \Si}|\nabla H|^2
+\sqrt c_b\sqrt\d r^{-\frac12}\left(r^{-1}\int_{B_r\cap \Si}|\nabla H|^2\right)^{1/2}\bigg).\end{split}\end{equation}
Using Remark \ref{c2 estimate}, with $f=|\nabla H|^2$ we have
\[
r^{-1}\int_{B_r\cap \Si}|\nabla H|^2\le 2\int_{\Si} |\nabla H|^2+\frac 12\int_{\Si}|\nabla |\nabla H|^2|
 \le  c\left(\int_{\Si} |\nabla H|^2+ \int_{\Si} |\nabla^2 H|^2\right).
\]
Therefore we have shown that for any $ r\in(0,1]$
\[
\frac12r^{-3}\int_{B_r\cap \Si}(r^2-|x|^2) (c|\nabla H|^2-2 h_{ij}\nabla_i\nabla_jH)\le
 c\left(\|H\|_{W^{2,2}(\Si)}
+\sqrt c_br^{-\frac12}\sqrt{\d}\|H\|^{\frac12}_{W^{2,2}(\Si)}\right).
 \]
 Here $\|H\|_{W^{2,2}(\Si)}$, denotes the $W^{2,2}$ norm of $H$ on $\Si$, i.e.
 \[\|H\|_{W^{2,2}(\Si)}:=\int_{\Si}|H|^2+\int_{\Si}|\nabla H|^2+\int_{\Si}|\nabla^2 H|^2.\]
Applying Lemma \ref{GMVin} we obtain
\begin{equation}\label{A0 bound}
\begin{split}
\pi|A(0)|^2&\le e^{\sqrt2 +1}\left(\int_{B_1\cap\Si} |A|^2+2c\left(\sqrt c_b\sqrt{\d}\|H\|^{\frac12}_{W^{2,2}(\Si)}+\|H\|_{W^{2,2}(\Si)}\right)\right)\\
&\le e^{\sqrt2 +1}(\d\e_0+2c\d(\sqrt c_b\sqrt{\e_0}+\e_0))\le c\d\sqrt{\e_0}
\end{split}
\end{equation}
where $c$ is an absolute constant. Thus, we can pick $\e_0$ sufficiently small, so that $|A(0)|^2<\frac{\d}{64}$, which contradicts \eqref{assu}. This finishes the proof of case (i).

Note that in equation \eqref{A0 bound} we have used the following elementary inequality to estimate $\|H\|_{W^{2,2}(\Si)}$. With an abuse of notation, let us reintroduce the tildes to denote the surfaces and quantities obtained after rescaling so that $\|H\|_{W^{2,2}(\Si)}=\|\wt{H}\|_{W^{2,2}(\wt{\Si}_0)}$ and $\wt{\Si}=\eta_{x_0,\frac{\s}{4}}(\B_{r_0})$. Then by the definition of the rescale invariant norms it follows that
\begin{equation}\label{rescaling}
\|\wt{H}\|_{W^{2,2}(\wt{\Si}_0)}\leq\|\wt{H}\|^*_{W^{2,2}(\wt{\Si})}=\|H\|^*_{W^{2,2}(\B_{r_0})}\le\d\e_0.
\end{equation}

We now prove case (ii), i.e. when we assume that $\|H\|^*_{W^{1,p}(\B_{r_0})}\le(\delta\e_0)^{p/2}$. We note that the same argument, as in case (i), carries through up to inequality \eqref{badterm}. In this case, instead of using Remark \ref{c2 estimate}, we will bound the RHS of \eqref{badterm} using the area bound \eqref{areabd}. In particular we obtain
\[\begin{split}r^{-1}\int_{B_r\cap \Si}|\nabla H|^2
&\le r^{-1} \area(\Si\cap B_r(x))^{1-2/p}\left(\int_{B_r\cap \Si}|\nabla H|^p\right)^{2/p}\\
&\le c^{1-2/p}r^{1-4/p}{\|H\|_{W^{1,p}(\Si)}}^{2/p} \le c{\|H\|_{W^{1,p}(\Si)}}^{2/p} r^{-2/p} \end{split}\]
where $c$ is an absolute constant (independent of $p$) and the last inequality is true since $r\le 1$ and $p>2$. Here $\|H\|_{W^{1,p}(\Si)}$ denotes the $W^{1,p}$ norm of $H$ in $\Si$, i.e.
\[\|H\|_{W^{1,p}(\Si)}:=\int_{\Si}|H|^p+\int_{\Si}|\nabla H|^p.\]
Using this estimate in \eqref{badterm} we get

\begin{equation*}\label{bad term}\begin{split}
\frac12r^{-3}\int_{B_r\cap \Si}(r^2-|x|^2)&(c|\nabla H|^2-2 h_{ij}\nabla_i\nabla_jH)\le\\
  &c \left({\|H\|_{W^{1,p}(\Si)}}^{2/p} r^{-2/p}
+\sqrt\d r^{-\frac12}r^{\frac12-\frac2p}{\|H\|_{W^{1,p}(\Si)}}^{1/p}\right)\le \\ &c r^{-2/p} \left({\|H\|_{W^{1,p}(\Si)}}^{2/p}
+\sqrt\d {\|H\|_{W^{1,p}(\Si)}}^{1/p}\right).\end{split}\end{equation*}

Applying Lemma \ref{GMVin} in this case with $\a=2/p$, we therefore obtain

 \begin{equation*}
\begin{split}
\pi|A(0)|^2&\le e^{\sqrt2 +1}\left(\int_{B_1\cap\Si} |A|^2+c\frac {p}{p-2}\left({\|H\|_{W^{1,p}(\Si)}}^{2/p}
+\sqrt\d {\|H\|_{W^{1,p}(\Si)}}^{1/p}\right)\right)\\
&\le e^{\sqrt2 +1}\left(\d\e_0+c\frac {p}{p-2}\d(\sqrt{\e_0}+\e_0)\right)\le c\d\sqrt{\e_0}\left(1+\frac{p}{p-2}\right)
\end{split}
\end{equation*}

Thus, we can pick $\e_0$ sufficiently small, depending on $p$, so that $|A(0)|^2<\frac{\d}{64}$, which contradicts \eqref{assu}. This finishes the proof of Theorem~\ref{choischoen}. Note that to estimate $\|H\|_{W^{1,p}(\Si)}$ we have also used the same argument as in equation~\eqref{rescaling}.
\qed

\begin{remark}\label{rmk}
The hypotheses needed to generalize the Choi-Schoen curvature estimate are optimal. For some $\a\in(0,1/2)$ and $\e\in(0,1)$, let
\[
u_\e(x,y)=xy\frac{\log^\a (x^2+y^2+\e)}{\log^\a\e}\,
\]
over the disk centered at the origin of radius $1/2$.
Let $\{\e_i\}\subset(0,1)$ be a sequence such that $\e_i\to 0$. Then the graphs of the functions $u_{\e_i}$ provide a sequence of surfaces $\Sigma_{\e_i}$ for which
\[\|A\|_{L^2(\Sigma_{\e_i})}\to 0\,,\, \|H\|_{W^{1,2}(\Sigma_{\e_i})}\to 0\]
but
\[|A_{\Si_{\e_i}}|(0)\to 1.\]
Note that this example shows that the hypotheses for \emph{Theorem~\ref{main}} are also optimal.
\end{remark}

\section{Colding-Minicozzi curvature estimate generalized }\label{CMgen}

In this section we prove Theorem~\ref{main} with the additional assumption that the $L^2$ norm of the mean curvature is small, Theorem~\ref{varofmain}. The idea of the proof is essentially the one in~\cite{cm23} except that we need to keep track of the mean curvature and use the more general results proved in Section~\ref{chs}. 

\begin{lemma}\label{lem1} Given $C$ and $p\ge 2$, there exists $\e_1=\e_1(p, C)>0$ such that the following holds. Let $\Sigma$ be a surface embedded in $\mathbb{R}^3$ containing the origin and such that $\inj_\Sigma(0)\geq 9s$. If
\[\int_{\B_{9s}}|A|^2\le C, \quad \int_{\B_{9s}\setminus\B_s}|A|^2\le\e_1 , \quad \int_{\B_{9s}}|H|^2\le \e_1\]
and either
\begin{itemize}
\item[i.] $\|H\|^*_{W^{2,2}(\B_{9s})}\le\e_1$, if $p=2$,
 or
\item[ii.] $\|H\|^*_{W^{1,p}(\B_{9s})}\le\e_1^{p/2}$, if $p>2$,
\end{itemize}
then
\[\sup_{\B_s}|A|^2\le s^{-2}.\]
\end{lemma}

\begin{proof}
In order to prove the lemma, we are going to show that if $\e_1$ is sufficiently small (depending on $p$ and $C$) then
\[ \int_{\B_{2s}}|A|^2\le\e_0\]
and either
\begin{itemize}
\item[i.] $\|H\|^*_{W^{2,2}(\B_{2s})}\le\e_0$, if $p=2$,
 or
\item[ii.] $\|H\|^*_{W^{1,p}(\B_{2s})}\le\e_0^{p/2}$, if $p>2$,
\end{itemize}
where $\e_0=\e_0(p)$ is such that the conclusion of Theorem~\ref{choischoen} holds.  Consequently, applying Theorem~\ref{choischoen} proves the lemma.

Note first that (i) or (ii) are automatically true by the hypotheses, as long as $\e_1\le\e_0$ and therefore we only need to show the estimate for $\int_{\B_{2s}}|A|^2$.

By Theorem~\ref{choischoen}, for $\e_1$ sufficiently small
\begin{equation}\label{Aonsmallann}\sup_{\B_{8s}\setminus\B_{2s}}|A|^2\le C_1^2\e_1 s^{-2}\end{equation}
where $C_1=\frac{1}{\sqrt{\e_0}}$ is fixed and $\e_0$ is as above.

Since $\inj_\Sigma(0)\geq 9s$, using Gauss-Bonnet yields
\[\length(\partial\B_{2s})-4\pi s=-\int_0^{2s}\int_{\B_\r} K_\Si,\]
and thus we have
\[\length(\partial\B_{2s})\le4\pi s+s\int_{\B_{2s}} (|A|^2+|H|^2)\le(4\pi+C)s+s\int_{\B_{2s}} |H|^2.\]
Therefore
\[\diam(\B_{8s}\setminus\B_{2s})\le 6s+\frac{4\pi+C+1}{2}s +6s\le(13+2\pi+C/2)s \]
provided that $\e_1\le 1$.
Let $x, x'\in \B_{8s}\setminus\B_{2s}$ and let $\gamma=\gamma(t)$ be a path between them parametrized by arclength so that $\gamma\subset\B_{8s}\setminus\B_{2s}$ and $t_0=\length\gamma\le\diam(\B_{8s}\setminus\B_{2s})$. Then, by letting $n(x)$ denote the normal of $\Si$ at the point $x$, we have
\[\begin{split}\dist_{S^2}(n(x'), n(x))&=\int_0^{t_0}\frac{d}{dt}\dist_{S^2}(n(\gamma(t)), n(x))
\le \int_0^{t_0}|\nabla\dist_{S^2}(n(\gamma(t)), n(x))|\\&\le\int_0^{t_0}|A(\gamma(t))|
\le C_1\e_1^{1/2}(13+2\pi+C/2)\end{split}\]
and thus
\[\sup_{x,x'\in\B_{8s}\setminus\B_{2s}}\dist_{S^2}(n(x'), n(x))\le C_1\e_1^{1/2}(13+2\pi+C/2).\]
By rotating $\R^3$ so that $n(p)=e_3$ for some $p\in\B_{8s}\setminus\B_{2s}$ we then get
\begin{equation}\label{nablax3}\sup_{\B_{8s}\setminus\B_{2s}}|\nabla x_3|\le C_1\e_1^{1/2}(13+2\pi+C/2)\end{equation}
since for $x=(x_1,x_2,x_3)$
\[|\nabla x_3|=|\proj_{T_x\Si}(Dx_3)|=|\proj_{T_x\Si} (e_3)|=|n(x)-e_3|=|n(x)-n(p)|.\]

Given $y\in \partial\B_{2s}$, let $\gamma_y$ be the outward normal geodesic from $y$ to $\partial\B_{8s}$ parametrized by arclength on $[0,6s]$. Then, \eqref{Aonsmallann} implies that for $\e_1$ small enough, $|\gamma_y(6s)-\gamma_y(0)|>(5+\frac12)s$ (see also Lemma \ref{distlemma}) which in turn implies that
 \[|\Pi(\gamma_y(6s))-\Pi(\gamma_y(0))|>5s\]
 where $\Pi$ denotes the projection onto the $(x_1,x_2)$ plane.
 This last implication follows from \eqref{nablax3}, since
  \[|x_3(\gamma_y(6s))-x_3(\gamma_y(0))|\le\int_{\gamma_y|_{[0,6s]}}|\nabla x_3|\le 6sC_1\e_1^{1/2}(13+2\pi+C/2).\]

Let us denote by $C_r$ the vertical cylinder of radius $r$, $C_r:=\{x_1^2+x_2^2=r^2\}$. The previous discussion implies that the intersection between $C_{3s}$ and the boundary of the annulus $\B_{8s}\setminus \B_{2s}$ is empty. More precisely, $\partial \B_{2s}$ is contained inside the cylinder while $\partial \B_{8s}$ is outside. From this observation it follows that $C_{3s}\cap\{ \B_{8s}\setminus \B_{2s} \}$ consists of a collection of closed curves. Since $\B_{8s}\setminus \B_{2s}$ is locally graphical over $\{x_3=0\}$ and the surface is embedded, each curve is  a graph over $\partial D_{3s}$, where $D_r$ is the disk of radius $r$ centered at the origin in the $\{x_3=0\}$ plane. For each $y\in \partial \B_{2s}$  let $\overline{t}_y$ be the minimum $t>2s$ such that $\gamma_y(\overline{t}_y)\in C_{3s}$ and let $\Gamma$ be the curve in $C_{3s}\cap\{ \B_{8s}\setminus \B_{2s} \}$ defined by
$$\Gamma := \{\gamma_y(\overline{t}_y) \mid y\in\partial \B_{2s} \}.$$ Since $\inj_\Sigma(0)\geq 9s$, such $\Gamma$ is a deformation retract of $\partial \B_{2s}$ and thus it is the boundary of a disk $\Delta$ containing $\B_{2s}$.

Using Gauss-Bonnet and Gauss equation,
$$2\pi- \int_\Gamma k_g=\int_\Delta K_\Sigma=\frac12\int_\Delta (H^2-|A|^2)$$
and thus
\begin{equation}\label{gb1}
\int_{\B_{2s}}|A|^2\leq \int_{\Delta}|A|^2 \leq \int_{\Delta}H^2+2\int_\Gamma k_g -4\pi.
\end{equation}

For  $\Gamma$ we have that
\[\sup_{\Gamma}|A|\le C_1\e_1^\frac12 s^{-1} \quad \text{and} \quad \sup_{\Gamma}|\nabla x_3|\le C_1\e_1^{1/2}(13+2\pi+C/2).\]
Thus applying a standard argument, that is Lemma~\ref{appendix} in the Appendix with $\e=C_1\e_1^{1/2}(13+2\pi+C/2)$, we obtain that

$$\length(\Gamma)\leq 6\pi s(1+2\e) \quad \text{and}\quad |k_g|< (3s)^{-1}(1+c\e).$$
where $c$ is an absolute constant, and
\begin{equation}\label{gb2}\begin{split}
\int_\Gamma k_g -2\pi&\leq \length(\Gamma)\sup_\Gamma|k_g|-2\pi\leq 2\pi(1+2\e)(1+c\e)-2\pi\\
&\leq 6\pi(1+c)\e= 6\pi(1+c)C_1\e_1^{1/2}(13+2\pi+C/2).\end{split}
\end{equation}
Using $\int_\Delta H^2\le\e_1$, together with equations \eqref{gb1} and \eqref{gb2}, if $\e_1$ is sufficiently small we obtain that
$$\int_{\B_{2s}}|A|^2\leq \e_0$$ and applying the Choi-Schoen estimate generalized finishes the proof of the lemma.
\end{proof}

Using Lemma~\ref{lem1}  the proof of the following variation of the Theorem \ref{main} is rather straightforward. Theorem~\ref{varofmain} below is in fact Theorem \ref{main} with the additional assumption that the $L^2$ norm of the mean curvature is small.

\begin{theorem}\label{varofmain}
Given $C_1$ and $p\ge 2$, there exist $C_2=C_2(p, C_1)\ge0$ and $\e_p=\e_p(C_1)>0$ such that the following holds. Let $\Sigma$ be a surface embedded in $\mathbb{R}^3$ containing the origin with $\inj_\Sigma(0)\geq s >0$,
$$\int_{\B_{s}}|A|^2\le C_1\,,\,\int_{\B_{s}}|H|^2\le\e_p$$
and either
\begin{itemize}
\item[i.]  $ \|H\|^*_{W^{2,2}(\B_{s})}\le\e_2$, if $p=2$, or
\item[ii.] $\|H\|^*_{W^{1,p}(\B_{s})}\le\e_p$, if $p>2$,
\end{itemize}
then
\[
|A|^2(0)\le C_2s^{-2}.\]
\end{theorem}

Note that in case (i) of Theorem~\ref{varofmain} the assumption $\int_{\B_s}|H|^2\leq \e_2$ becomes redundant. Thus with Theorem~\ref{varofmain} we have proved case (i) of Theorem~\ref{main}. In order to prove case (ii) of Theorem~\ref{main}, we need to remove the extra assumption on the $L^2$ norm of $H$. This will be done in the next section.\\

\noindent{\textit{Proof of Theorem~\ref{varofmain}.}}
Let $\e_1=\e_1(p,C_1)$ be such that the conclusion of Lemma~\ref{lem1} holds with $C=C_1$ and let $N$ be the least integer greater than $\frac{C_1}{\e_1}$. Without loss of generality, let us assume $\e_1<1$. We are going to show that if $\e_p<\e_1^{p\slash2}\le\e_1$, then $C_2$ can be taken to be $9^{2N}$.

There exists $1\leq j\leq N$ such that
\[
\int_{\B_{9^{1-j}s}\backslash \B_{9^{-j}s}}|A|^2\le \frac{C_1}{N}\leq \e_1.
\]
Moreover
\[\int_{\B_{9^{1-j}s}}|H|^2\le \int_{\B_{s}}|H|^2\le\e_p< \e_1\]
and, in case $p=2$,
\[
\|H\|^*_{W^{2,2}(\B_{9^{1-j}s})}\leq\|H\|^*_{W^{2,2}(\B_{s})}\le\e_2<\e_1
\]
or, in case $p>2$,
\[
\|H\|^*_{W^{1,p}(\B_{9^{1-j}s})}\leq\|H\|^*_{W^{1,p}(\B_{s})}\le\e_p< \e_1^{p\slash2}.
\]
Thus, applying Lemma~\ref{lem1} to $\B_{9\frac{s}{9^{j}}}$, we obtain
\[
|A|^2(0)\leq \left(\frac{s}{9^{j}}\right)^{-2}\leq 9^{2N}s^{-2}.
\]\qed

\section{Total Curvature and Area}\label{elp}

In this section we prove case (ii) of Theorem~\ref{main} by removing the extra assumption on the $L^2$ norm of $H$ in Theorem \ref{varofmain}.  In order to do that we show that when the total curvature is bounded, then the rescale invariant $L^p$ norm of the mean curvature, $p>2$,  bounds the $L^2$ norm of the mean curvature. In addition, we then show that when the rescale invariant $L^p$ norm of the mean curvature is bounded, $p\geq2$, a bound on the total curvature of an intrinsic ball provides a bound for its area, and viceversa (see for instance~\cite{cm22} for this being done in the minimal case).
This relation  enables us to replace the bound on the total curvature in Theorem~\ref{main} with an area bound, Corollary~\ref{cor}.



An easy computation using Gauss-Bonnet Theorem (see for instance~\cite{cm22}) gives the following lemma.

\begin{lemma}
If $\B_s\subset \Sigma$ is such that  $\inj_\Sigma(0)\geq s>0$, then 
\[
\area\B_s-\pi s^2=-\int_0^s\int_0^t\int_{\B_\rho}K_\Sigma.
\]
\end{lemma}

Since $-2K_\Sigma=|A|^2-|H|^2$, the above lemma implies

\begin{equation}\label{area}
\area\B_s\leq \pi s^2 + \frac12s^2\int_{\B_s}|A|^2+\frac12s^2\int_{\B_s}|H|^2.
\end{equation}
Furthermore
\[
-\int_{\B_s} K_\Sigma \frac{(s-r)^2}{2}=-\int_0^s\int_0^t\int_{\B_\rho}K_\Sigma=\area{\B_s}-\pi s^2
\]
where the first equality follows by the coarea formula and integration by parts twice (cf.  Corollary 1.7 in~\cite{cm22})
and thus

\begin{equation}\label{AaH}
\begin{split}
\frac{s^2}{16}\int_{\B_{\frac{s}{2}}}|A|^2&\leq \int_{\B_s} \frac{|A|^2}{2} \frac{(s-r)^2}{2}= \area{\B_s}-\pi s^2+\int_{\B_s} \frac{|H|^2}{2} \frac{(s-r)^2}{2}\\
&=\area{\B_s}-\pi s^2+\int_0^s\int_0^t\int_{\B_\rho}\frac{|H|^2}{2}
\leq \area{\B_s}-\pi s^2
+\frac{s^2}{2}\int_{\B_s}|H|^2.
\end{split}
\end{equation}
Here for any $x\in\B_s$, $r=r(x)$ denotes the geodesic distance from the origin.
Hence \eqref{area} and \eqref{AaH} show that when $\int_{\B_s}H^2$ is bounded, a bound on the total curvature provides a bound on the area and viceversa.

In the following lemma we show that a bound on  $s^{p-2}\int_{\B_s}|H|^p$, $p>2$, and either  a bound on the total curvature or a bound on the area, provide a bound on the $L^2$ norm of $H$. Furthermore when the bound of $s^{p-2}\int_{\B_s}|H|^p$ is small, the  $L^2$ norm of $H$ is also small.
\begin{lemma}\label{H2} If $\B_s\subset \Sigma$ is such that  $\inj_\Sigma(0)\geq s>0$ and  $\left(s^{p-2}\int_{\B_s}|H|^p\right)^{1/p}=\e$. Then for $q$ such that $1/q+2/p=1$
\begin{equation}\label{area2}
\int_{\B_s}|H|^2\le \e^2 (s^{-2}\area{\B_s})^{1/q}
\end{equation}
 and 
 \begin{equation}\label{fp}
\int_{\B_s}|H|^2\le 2^{p/2} (\e^2+\e^p)\left(\pi+\frac12\int_{\B_s}|A|^2\right)^{1/q}.
\end{equation}
\end{lemma}
 
 \begin{remark}\label{rmk2}Equation \eqref{fp} shows that if $\int_{\B_s}|A|^2$ is bounded, then if the $L^p$ norm of $H$ is small, $p>2$, so is the $L^2$ norm of $H$. This estimate, together with Theorem~\ref{varofmain} proves Theorem~\ref{main}.
\end{remark}

\noindent\textit{Proof of Lemma \ref{H2}.}
The proof of inequality \eqref{area2} is just an application of the Holder inequality. Using \eqref{area} and the Minkowski inequality to estimate the RHS of \eqref{area2} we get

\[\begin{split}\int_{\B_s}|H|^2\le&\e^2\left(\pi+\frac12\int_{\B_s}|A|^2+\frac12\int_{\B_s}|H|^2\right)^{1/q}\le
\e^2\left(\pi+\frac12\int_{\B_s}|A|^2\right)^{1/q}+\e^2\left(\frac12\int_{\B_s}|H|^2\right)^{1/q}.\end{split}\]
If
\begin{equation}\label{cond}{\e^2}\left(\frac12\int_{\B_s}|H|^2\right)^{1/q}\le \frac12\int_{\B_s}|H|^2\end{equation}
then the previous inequality gives
\[\int_{\B_s}|H|^2\le 2 \e^2\left(\pi+\frac12\int_{\B_s}|A|^2\right)^{1/q}.\]
On the other hand, if \eqref{cond} does not hold then
\[{\e^2}\left(\frac12\int_{\B_s}|H|^2\right)^{1/q}> \frac12\int_{\B_s}|H|^2\implies\left(\int_{\B_s}|H|^2\right)^{1-1/q}\le2\e^2\implies\]
\[\int_{\B_s}|H|^2\le 2^{p/2}\e^p.\]
Therefore in either case \eqref{fp} is true.
\qed

Using \eqref{fp} and \eqref{area2} of Lemma \ref{H2} to estimate the RHS of \eqref{area} and \eqref{AaH} respectively we have the following two estimates that show that when the rescale invariant $L^p$ norm of the mean curvature is bounded, $p>2$, a bound on the total curvature of an intrinsic ball provides a bound for its area, and viceversa

\begin{equation}\label{area3}
\area\B_s\leq \left(1+\e^22^{\frac p2-1}\right)\left(\pi s^2 + \frac12s^2\int_{\B_s}|A|^2\right)
\end{equation}
and
\begin{equation}\label{forcor}\frac{s^2}{16}\int_{\B_{\frac{s}{2}}}|A|^2\leq \area{\B_s}-\pi s^2+\frac{\e^2}{2}s^\frac{4}{p}\area\B_s^{1/q}
\end{equation}
Finally, \eqref{forcor}  implies that Theorem~\ref{main} still holds if instead of a bound on the total curvature, we assume a bound on the area and thus we have the following corollary of Theorem~\ref{main}.

\begin{corollary}\label{cor}
Given $C_1$ and $p\ge 2$, there exist $C_2=C_2(p, C_1)\ge0$ and $\e_p=\e_p(C_1)>0$ such that the following holds. Let $\Sigma$ be a surface embedded in $\mathbb{R}^3$ containing the origin with $\inj_\Sigma(0)\geq s >0$,
$$\min \{ s^{-2}\area(\B_{s}), \int_{\B_{s}}|A|^2\}\le C_1$$
and either
\begin{itemize}
\item[i.]  $ \|H\|^*_{W^{2,2}(\B_{r_0})}\le\e_2$, if $p=2$, or
\item[ii.] $\|H\|^*_{W^{1,p}(\B_{r_0})}\le\e_p$, if $p>2$,
\end{itemize}
then
\[
|A|^2(0)\le C_2s^{-2}.\]
\end{corollary}

\section{Appendix}

In this Appendix we review some standard geometric facts about surfaces with bounded second fundamental form that are needed in the paper.

\begin{lemma}\label{distlemma}
Let $\Si$ be a surface in $\R^3$, $p,q\in\Si$ and let $\gamma:[0,\l]\to\Si$ be a geodesic, parametrized by arclength, such that $\gamma(0)=p$ and $\gamma(\l)=q$. If for some $\a\geq0$,
\[\sup_{t\in[0,\l]}|A(\gamma(t))|\leq\frac\a\l\]
then $|q-p|\geq\l(1-\a)$.

\end{lemma}
\begin{proof}
Let $k$ denote the curvature of  $\gamma$ in $\R^3$. Then, since $\gamma$ is a geodesic, for any $t\in[0,\l]$
\[|k(t)|\le |A(\gamma(t))|\leq \frac {\a}{\l}.\]
Since
\[\left|\frac{d}{dt}\langle \gamma'(t),\gamma'(0)\rangle \right|\le |k|\]
we have for all $t_0\in[0,\l]$
\[\langle \gamma'(t_0),\gamma'(0)\rangle-1=\int_0^{t_0}\frac{d}{dt}\langle \gamma'(t),\gamma'(0)\rangle \implies\]
\[\langle \gamma'(t_0),\gamma'(0)\rangle \ge 1-\int_0^{t_0}\left|\frac{d}{dt}\langle \gamma'(t),\gamma'(0)\rangle \right|\ge 1-\int_0^{t_0} |k(\gamma(t))|\ge 1- \a\frac {t_0}{\l}\ge (1- \a).\]
Also
\[\langle \gamma(\l), \gamma'(0)\rangle -\langle \gamma(0), \gamma'(0)\rangle =\int_0^{\l}\frac{d}{dt}\langle \gamma(t), \gamma'(0)\rangle \implies\]
\[\langle q-p, \gamma'(0)\rangle=\langle \gamma(\l)-\gamma(0), \gamma'(0)\rangle \ge \l(1- \a).\]
This implies that $|q-p|\geq\l(1-\a)$.
%
\end{proof}

\begin{lemma}\label{appendix}
There exist $c>0$ and  $\overline{\e}>0$ such that the following holds. Let $\Sigma$ be a surface in $\R^3$ that is graphical over some domain in the $\{x_3=0\}$  plane containing $\partial D_s$, for some $s>0$, and let $\Gamma$ be $\Sigma \cap C_s$. If for a certain $0\leq \e \leq \overline{\e}$,
\begin{equation}\label{app}
\sup_{\Gamma}|A|\le \e s^{-1} \quad \text{and} \quad \sup_{\Gamma}|\nabla x_3|\le \e
\end{equation}
then $$\length\Gamma\leq 2\pi s(1+2\e) \quad \text{and}\quad |k_g|< s^{-1}(1+c\e).$$
\end{lemma}
\begin{proof}
Let us consider the following parameterization for $\Gamma$,
\[\Gamma=r(t):= \{(s\cos( t/s), s\sin(t/s), x_3(t))\mid t\in[0,2\pi s)\}\]
Because of the second estimate in \eqref{app}
\[\frac{|\dot r\cdot e_3|}{|\dot r|}\le \proj_{T_{r}M}(e_3)\le\e\]
 and since $\dot r(t)=(-\sin(t/s),\cos(t/s),\dot x_3(t))$ we get
\begin{equation}\label{dot}\frac{|\dot x_3|}{\sqrt{1+|\dot x_3|^2}}\le\e\implies |\dot x_3|^2\le\frac{\e^2}{1-\e^2}\implies |\dot x_3|\le 2\e\end{equation}
for $\e<\sqrt3 /2$.
Thus for  $L=\length \Gamma$ we have the following:
\[2\pi s\leq L:=\int_0^{2\pi s}\sqrt{1+|\dot{x}_3|^2}\le 2\pi s(1+2\e).\]
Moreover
\[\ddot r(t)=(-s^{-1}\cos(t/s),-s^{-1}\sin(t/s),\ddot x_3(t))\]
and the curvature vector $k=(k_1,k_2,k_3)$ is given by the formula
\[k=\ddot r\frac{1}{1+\dot x_3^2}+\dot r\frac {\dot x_3\ddot x_3}{(1+\dot x_3^2)^2}.\]

Let $n=(n_1,n_2,n_3)$ be the unit normal of $\Si$, then by the second estimate in \eqref{app} we have that
\[\max\{|n_1|,|n_2|\}\le\e\,\,,\,\,n_3>1-\e.\] Then
\begin{equation}\label{nk}(1-\e)|k_3| \le |n_3  k_3| \le |n\cdot k|+|(n_1,n_2,0)\cdot  k|\le|A(r)|+\e|(k_1,k_2,0)|.\end{equation}
Note that
\begin{equation}\label{k3}k_3=\frac{\ddot x_3}{1+\dot x_3^2}+\frac{\ddot x_3\dot x_3^2}{(1+\dot x_3^2)^2}\implies |k_3|\ge\frac 12|\ddot x_3|\end{equation}
because of (\ref{dot}) and for $\e$ sufficiently small ($(1+\e^2)^{-2}\ge 1/2$).

Let $\a=\frac{1}{1+\dot x_3^2(t)}$    and    $\beta=\frac {\dot x_3(t)\ddot x_3(t)}{(1+\dot x_3^2(t))^2}$

\begin{equation}\label{k1k2}(k_1, k_2,0)=(-\frac{\a}{s}\cos(t/s)-\beta\sin(t/s),-\frac{\a}{s}\sin(t/s)+\beta\cos(t/s),0)\end{equation}
and
\[|(k_1, k_2,0)|\le\left(\frac{\a^2}{ s^2}+\beta^2\right)^{1/2}\le\left(s^{-2}+4\e^2\ddot x_3^2\right)^{1/2}\le s^{-1}+2\e|\ddot x_3|.\]
Therefore using equations \eqref{app}, \eqref{nk}, \eqref{k3} and \eqref{k1k2} we have that
\[\frac{1-\e}{2}|\ddot x_3|\le2\e s^{-1}+2\e^2|\ddot x_3|\implies |\ddot x_3|\le 16\e s^{-1}\]
for $\e$ sufficiently small ($\e<1/4$).

Therefore there exists an absolute constant $c$ such that for the geodesic curvature we have
\[|k_g|\le|\vec k|\le|\ddot r|+c\e^2s^{-1}|\dot r|\le \sqrt{s^{-2}(1+c\e^2)}+c\e^2s^{-1}\sqrt{1+c\e^2}\le s^{-1}(1+c\e)\]
if $\e$ is sufficiently small.
\end{proof}

\vspace{0in}
\bibliography{bill}
\bibliographystyle{plain}

\end{document}